\documentclass[12pt,a4paper]{smfart}
\usepackage{smfenum}
\usepackage{amsmath,amssymb,amsthm,latexsym,MnSymbol}
\usepackage[utf8]{inputenc}
\usepackage{url}
\usepackage[pdftex]{graphicx}
\usepackage{tikz}
\usepackage[pdftex]{hyperref}
\usepackage[all]{xy}
\usepackage[margin=2cm]{geometry}

\newcommand{\Z}{\mathbf{Z}}
\newcommand{\Q}{\mathbf{Q}}

\newcommand{\C}{\mathbf{C}}
\newcommand{\h}{\mathfrak{H}}
\newcommand{\Gm}{\mathbf{G}_m}
\newcommand{\eps}{\varepsilon}
\newcommand{\epsb}{{\overline{\varepsilon}}}
\newcommand{\chib}{{\overline{\chi}}}

\DeclareMathOperator{\dv}{div}

\DeclareMathOperator{\Gal}{Gal}

\DeclareMathOperator{\Hom}{Hom}

\DeclareMathOperator{\imag}{Im}

\DeclareMathOperator{\real}{Re}

\DeclareMathOperator{\SL}{SL}

\providecommand{\abs}[1]{|#1|}

\newtheorem{thm}{Theorem}
\newtheorem{lem}[thm]{Lemma}
\newtheorem{pro}[thm]{Proposition}

\theoremstyle{definition}
\newtheorem{definition}[thm]{Definition}

\theoremstyle{remark}
\newtheorem{remark}[thm]{Remark}

\newtheorem{question}[thm]{Question}

\author{François Brunault}
\email{francois.brunault@ens-lyon.fr}
\address{ÉNS Lyon, UMPA, 46 allée d'Italie, 69007 Lyon, France}

\begin{document}

\renewcommand{\refname}{References}

\mainmatter

\begin{center}
\Large
\textbf{On the Mahler measure associated to $X_1(13)$}
\normalsize
\end{center}

\vspace{.5cm}

\emph{Abstract.}
We show that the Mahler measure of a defining equation of the modular curve $X_1(13)$ is equal to the derivative at $s=0$ of the $L$-function of a cusp form of weight 2 and level 13 with integral Fourier coefficients. The proof combines Deninger's method, an explicit version of Beilinson's theorem together with an idea of Merel to express the regulator integral as a linear combination of periods. Finally, we present further examples related to the modular curves of level 16, 18 and 25.

\vspace{.5cm}

The Mahler measure of a polynomial $P \in \C[x_1^{\pm 1},\ldots,x_n^{\pm 1}]$ is defined by
\begin{equation*}
m(P)=\frac{1}{(2\pi i)^n} \int_{|z_1|=1} \cdots \int_{|z_n|=1} \log |P(z_1,\ldots,z_n)| \frac{dz_1}{z_1} \cdots \frac{dz_n}{z_n}.
\end{equation*}

In a fascinating paper, Boyd \cite{boyd:expmath} developed a body of conjectures relating Mahler measures of $2$-variable polynomials and special values of $L$-functions of elliptic curves. Deninger \cite{deninger:mahler} provided a bridge between the world of Mahler measures and certain $K$-theoretic regulators. He thus showed the relevance of Beilinson's conjectures to prove relations between Mahler measures and special values of $L$-functions. In the case of curves, such identities have been proven rigorously only in rare instances, mainly in the case of genus $0$ and $1$ (\cite{brunault:cras}, \cite{rogers-zudilin}, \cite{rogers-zudilin2}...). There has been some recent work, however, in the case of genus 2 (see \cite{bertin-zudilin} and the references therein).

The aim of this paper is to achieve such a relation in the case of a curve of genus $2$. We work with the modular curve $X_1(13)$. Thanks to \cite[p. 56]{lecacheux:13}, a defining equation of $X_1(13)$ is
\begin{equation*}
P=y^2 x(x-1)+y(-x^3+x^2+2x-1)-x^2+x.
\end{equation*}

We prove the following theorem.

\begin{thm}\label{main thm}
We have the identity $m(P)=2L'(f,0)$, where $f$ is the cusp form of weight $2$ and level $13$ whose Fourier expansion begins with
\begin{equation*}
f = 2q-3q^2-2q^3+q^4+6q^6-q^9-3q^{10}-4q^{12}-5q^{13}+O(q^{15}).
\end{equation*}
\end{thm}
Note that the cusp form $f$ is not a newform; rather, it is the trace of the unique (up to Galois conjugacy) newform of weight $2$ on the group $\Gamma_1(13)$.

In the last section, we present further examples of relations between Mahler measures and $L$-values in the case of the modular curves $X_1(16)$, $X_1(18)$ and $X_1(25)$.

This article grew out of results in my PhD thesis (see especially \cite[\S 3.8 and Remarque 112]{brunault:these}). I would like to thank Odile Lecacheux for helpful exchanges having led to the discovery of these identities. I would also like to thank Wadim Zudilin for useful comments.

\section{Deninger's method}

In this section we express the Mahler measure of $P$ in terms of the integral of a differential $1$-form on the modular curve $X_1(13)$, following Deninger's method \cite{deninger:mahler}.

We view $P$ as a polynomial in $h$:
\begin{equation*}
P(H,h) = -H + (-H^2+2H+1)h+(H^2+H-1)h^2-Hh^3.
\end{equation*}
Note that the constant term of $P$ is given by $P^*(H)=-H$.

Let $Z \subset \Gm^2$ be the curve defined by the equation $P=0$. Then $Z$ identifies with an affine open subscheme of $X_1(13)$ by \cite[p. 56]{lecacheux:13}. In particular $Z$ is smooth.

Looking at the resultant of the polynomials $P(H,h)$ and $H^2 h^3 P(\frac{1}{H},\frac{1}{h})$ with respect to $h$, it can be checked that $P$ doesn't vanish on the torus $T^2 = \{(H,h) \in \C : |H|=|h|=1\}$. Moreover, we check numerically that for each $H \in T$, there exists a unique $h(H) \in \C$ such that $P(H,h(H))=0$ and $0<|h(H)|<1$. The map $H \in T \mapsto h(H)$ defines a closed cycle $\gamma_P$ in $H_1(Z(\C),\Z)$. We call $\gamma_P$ the \emph{Deninger cycle} associated to $P$. We give $\gamma_P$ the canonical orientation coming from $T$.

Since $P^*$ doesn't vanish on $T$, the polynomial $P$ satisfies the assumptions \cite[3.2]{deninger:mahler}, so that the discussion in \emph{loc. cit.} applies. Consider the differential form $\eta = \log |h| \frac{dH}{H}$ on $Z(\C)$. Using Jensen's formula, and noting that $m(P^*)=0$, we have \cite[(23)]{deninger:mahler}
\begin{equation*}
m(P) = -\frac{1}{2\pi i} \int_{\gamma_P} \eta.
\end{equation*}
Now we may express this as an integral of a closed differential form. By \cite[Prop. 3.3]{deninger:mahler}, we get
\begin{equation*}
m(P) = -\frac{1}{2\pi i} \int_{\gamma_P} \log |H| \cdot (\partial-\overline{\partial}) \log |h| - \log |h| \cdot (\partial-\overline{\partial}) \log |H|.
\end{equation*}

We now introduce a standard notation.

\begin{definition}
For any two meromorphic functions $u,v$ on a Riemann surface, define
\begin{equation*}
\eta(u,v) : = \log |u| \operatorname{darg}(v) - \log |v| \operatorname{darg}(u).
\end{equation*}
\end{definition}

The $1$-form $\eta(u,v)$ is well-defined outside the set of zeros and poles of $u$ and $v$. It is closed, so we may integrate it over cycles. Moreover, we have $\operatorname{darg}(u) = -i (\partial-\overline{\partial}) \log |u|$. Thus we have proved the following proposition.


\begin{pro}\label{pro deninger}
We have $m(P)=\frac{1}{2\pi} \int_{\gamma_P} \eta(h,H)$.
\end{pro}

\begin{lem}
Let $c$ denote complex conjugation on $Z(\C)$. We have $c_* \gamma_P = -\gamma_P$.
\end{lem}

\begin{proof}[Proof]
For every $H \in T$, we have $h(\overline{H})=\overline{h(H)}$. It follows that $c_* \gamma_P=-\gamma_P$.
\end{proof}

\section{Determining Deninger's cycle}

In this section, we determine $\gamma_P$ explicitly in terms of modular symbols.

The space $S_2(\Gamma_1(13))$ of cusp forms of weight $2$ and level $13$ has dimension $2$ over $\C$. Let $\eps : (\Z/13\Z)^\times \to \C^\times$ be the unique Dirichlet character satisfying $\eps(2)=\zeta_6 := e^{\frac{2\pi i}{6}}$. It is even and has order $6$. A basis of $S_2(\Gamma_1(13))$ is given by $(f_\eps,f_\epsb)$, where $f_\eps$ (resp. $f_\epsb$) is a newform having character $\eps$ (resp. $\epsb$). The Fourier coefficients of $f_\eps$ and $f_\epsb$ belong to the field $\Q(\zeta_6)$ and are complex conjugate to each other. We define $f=f_\eps+f_\epsb$.

We denote by $\langle d \rangle$ the diamond automorphism of $X_1(13)$ associated to $d \in (\Z/13\Z)^\times/\pm 1$.

Let $\hat{\mathcal{H}}=H_1(X_1(13)(\C),\{\mathrm{cusps}\},\Z)$ be the homology group of $X_1(13)(\C)$ relative to the cusps. Let $E_{13}$ be the set of non-zero vectors of $(\Z/13\Z)^2$. For any $x \in E_{13}$, we let $\xi(x) = \{g_x 0, g_x \infty\}$, where $g_x \in \SL_2(\Z)$ is any matrix whose bottom line is congruent to $x$ modulo $13$. Using Manin's algorithm \cite{manin} and its implementation in Magma \cite{magma}, we find that a $\Z$-basis of $\mathcal{H}=H_1(X_1(13)(\C),\Z)$ is given by
\begin{align*}
\gamma_1 & = \xi(1,-5)-\xi(2,5)-\xi(1,-2) = \left\{\frac15,\frac25\right\}\\
\gamma_2 & = \langle 2 \rangle_* \gamma_1 = \xi(2,3)-\xi(4,-3)-\xi(2,-4)\\
\gamma_3 & = \xi(1,-3)-\xi(1,3) = \left\{\frac13,-\frac13\right\}\\
\gamma_4 & = \langle 2 \rangle_* \gamma_3 = \xi(2,-6)-\xi(2,6).
\end{align*}

Consider the pairing
\begin{align*}
\langle \cdot,\cdot \rangle : \hat{\mathcal{H}} \times S_2(\Gamma_1(13)) & \to \C\\
(\gamma,f) & \mapsto \int_\gamma 2\pi i f(z)dz.
\end{align*}

\begin{definition}
Let $\mathcal{H}^- := \{ \gamma \in \mathcal{H} : c_* \gamma=-\gamma\}$. We define the map
\begin{align*}
\iota : \mathcal{H}^- & \to \C\\
\gamma & \mapsto \langle \gamma,f_\varepsilon \rangle.
\end{align*}
\end{definition}

\begin{lem} \label{iota injective}
The map $\iota$ is injective.
\end{lem}

\begin{proof}[Proof]
If $\iota(\gamma)=0$ then $\langle \gamma,f_{\overline{\varepsilon}} \rangle = \overline{\langle c_* \gamma, f_\varepsilon \rangle} = -\overline{\langle \gamma,f_\varepsilon \rangle} = 0$. Thus $\gamma$ is orthogonal to $S_2(\Gamma_1(13))$, which implies $\gamma=0$.
\end{proof}

\begin{lem}\label{iota lattice}
The image of $\iota$ is the hexagonal lattice generated by $\iota(\gamma_3)$ and $\iota(\gamma_4) = \zeta_6 \iota(\gamma_3)$.
\end{lem}

\begin{proof}[Proof]
The action of complex conjugation on $\mathcal{H}$ is given by
\begin{align*}
c_*(\gamma_1) & = \gamma_1+\gamma_4\\
c_*(\gamma_2) & = \gamma_2 - \gamma_3 + \gamma_4\\
c_*(\gamma_3) & = -\gamma_3\\
c_*(\gamma_4) & = -\gamma_4.
\end{align*}
From these formulas, it is clear that a $\Z$-basis of $\mathcal{H}^-$ is given by $(\gamma_3,\gamma_4)$. By Lemma \ref{iota injective}, we have $\iota(\gamma_3) \neq 0$. Then
\begin{equation*}
\iota(\gamma_4) = \langle \langle 2 \rangle_* \gamma_3, f_\varepsilon \rangle = \langle \gamma_3, f_\varepsilon | \langle 2 \rangle \rangle = \varepsilon(2) \iota(\gamma_3) = \zeta_6 \iota(\gamma_3).
\end{equation*}
\end{proof}

We have $\gamma_3 = \{\frac13,-\frac13\} = \{\frac13,g_1 \left(\frac13\right)\}$ with $g_1 = \begin{pmatrix} 14 & -5 \\ -39 & 14 \end{pmatrix} \in \Gamma_1(13)$. Let us choose $z_0 = \frac{14+i}{39}$. Then 
$g_1 (z_0) = \frac{-14+i}{39}$. We have
\begin{equation*}
\langle \gamma_3,f_\varepsilon \rangle = \int_{z_0}^{g_1 z_0} 2\pi i f_\varepsilon(z) dz = \sum_{n=1}^{\infty} \frac{a_n(f_\varepsilon)}{n} \left(e^{\frac{-28\pi i n}{39}}-e^{\frac{28\pi i n}{39}}\right) e^{-\frac{2\pi n}{39}}.\end{equation*}
Using Magma, we get numerically
\begin{equation*}
\langle \gamma_3,f_\varepsilon \rangle \sim 1.06759 - 2.60094i.
\end{equation*}

\begin{pro} \label{pro gamma0}
Let $\gamma_P \in \mathcal{H}^-$ be Deninger's cycle. We have $\gamma_P=\gamma_3$.
\end{pro}

\begin{proof}[Proof]
A $\Q$-basis of $\Omega^1(X_1(13))$ is given by $(\omega, h\omega)$ where
\begin{equation*}
\omega = \frac{(h^2 - h)H - h^3 + h^2 + 2h - 1}{h^4 - 2h^3 + 3h^2 - 2h + 1}dH.
\end{equation*}
Using Magma, we compute the Fourier expansion of $\omega$ and $h\omega$ at infinity, and deduce
\begin{equation}\label{feps omega}
2\pi i f_\varepsilon(z)dz = \alpha \omega + \beta h\omega
\end{equation}
with
\begin{equation*}
\alpha \sim 0.71163+0.70256i \qquad \beta \sim 0.25262-0.96757i.
\end{equation*}
Note that $\alpha$ and $\beta$ are algebraic numbers, but we won't need an explicit formula for them. With Pari/GP \cite{pari273}, we find
\begin{equation}\label{int gammaP}
\int_{\gamma_P} \omega \sim - 3.21731i \qquad \int_{\gamma_P} h\omega \sim - 1.23275i.
\end{equation}
From (\ref{feps omega}) and (\ref{int gammaP}), it follows that
\begin{equation*}
\langle \gamma_P, f_\varepsilon \rangle \sim 1.06759-2.60094i \sim \langle \gamma_3, f_\varepsilon \rangle.
\end{equation*}
Since the image of $\iota$ is a lattice by Lemma \ref{iota lattice}, we may ascertain that $\gamma_P=\gamma_3$.
\end{proof}

We will also need to make explicit the action of the Atkin-Lehner involution $W_{13}$ on $\gamma_P$.

\begin{pro} \label{pro W13gamma0}
We have $W_{13} \gamma_P = \gamma_4-\gamma_3$.
\end{pro}

\begin{proof}[Proof]
By \cite[Thm 2.1]{Atkin-Li}, we have $W_{13} f_\varepsilon = w \cdot f_{\overline{\varepsilon}}$ with
\begin{equation}\label{eq w}
w = \frac{3\zeta_6-4}{13} \tau(\varepsilon) \sim -0.96425+0.26501i.
\end{equation}
We deduce
\begin{equation*}
\iota(W_{13} \gamma_P) = \langle \gamma_P, W_{13} f_\varepsilon \rangle = w \langle \gamma_P, f_{\overline{\varepsilon}} \rangle = w \overline{\langle c_* \gamma_P, f_\varepsilon \rangle} = -w \overline{\langle \gamma_P, f_\varepsilon \rangle} \sim 1.71869+2.22503i.
\end{equation*}
Moreover, we have
\begin{equation*}
\iota(\gamma_4) = \zeta_6 \iota(\gamma_3) \sim 2.78628-0.37591i \sim \iota(W_{13} \gamma_P) + \iota(\gamma_3).
\end{equation*}
Using Lemma \ref{iota lattice} again, we conclude that $W_{13} \gamma_P = \gamma_4-\gamma_3$.
\end{proof}

\section{Beilinson's theorem}

We now recall the explicit version of Beilinson's theorem on the modular curve $X_1(N)$ \cite{brunault:smf}. Let $\C(X_1(N))$ be the function field of $X_1(N)$. The \emph{regulator map} on $X_1(N)$ is defined by
\begin{align*}
r_N : K_2(\C(X_1(N))) & \to \Hom_\C(S_2(\Gamma_1(N)),\C)\\
\{u,v\} & \mapsto \left(f \mapsto \int_{X_1(N)(\C)} \eta(u,v) \wedge \omega_f \right)
\end{align*}
where $\omega_f := 2\pi i f(z) dz$. After tensoring with $\C$, we get a linear map
\begin{equation*}
r_N : K_2(\C(X_1(N))) \otimes \C \to \Hom_\C(S_2(\Gamma_1(N)),\C).
\end{equation*}
For any even non-trivial Dirichlet character $\chi : (\Z/N\Z)^\times \to \C^\times$, there exists a modular unit $u_{\chi} \in \mathcal{O}^{\times}(Y_1(N)(\C)) \otimes \C$ satisfying
\begin{equation*}
\log |u_{\chi}(z)| = \frac{1}{\pi} \lim_{\substack{s \rightarrow 1\\ \real(s)>1}} \left(\sideset{}{'}\sum_{(m,n) \in \Z^2} \frac{\chi(n) \cdot \imag(z)^s}{\abs{Nmz+n}^{2s}}\right) \qquad (z \in \h),
\end{equation*}
where $\sideset{}{'}\sum$ denotes that we omit the term $(m,n) = (0,0)$ (see \cite[Prop 5.3]{brunault:smf}).

\begin{remark}
We are working with the model of $X_1(N)$ in which the $\infty$-cusp is not defined over $\Q$, but rather over $\Q(\zeta_N)$. Therefore, the modular unit $u_\chi$ is not defined over $\Q$ but rather over $\Q(\zeta_N)$.
\end{remark}

\begin{thm}\cite[Thm 1.1]{brunault:smf}\label{explicit beilinson}
Let $f \in S_2(\Gamma_1(N),\psi)$ be a newform of weight $2$, level $N$ and character $\psi$. For any even primitive Dirichlet character $\chi : (\Z/N\Z)^\times \to \C^\times$, with $\chi \neq \overline{\psi}$, we have
\begin{equation}\label{explicit beilinson formula}
L(f,2) L(f,\chi,1) = \frac{N \pi \tau(\chi)}{2 \phi(N)} \bigl\langle r_N(\{u_{\overline{\chi}},u_{\psi\chi}\}), f \bigr\rangle
\end{equation}
where $L(f,\chi,s):=\sum_{n=1} a_n(f) \chi(n) n^{-s}$ denotes the $L$-function of $f$ twisted by $\chi$, $\tau(\chi):=\sum_{a \in (\Z/N\Z)^\times} \chi(a) e^{\frac{2\pi ia}{N}}$ denotes the Gauss sum of $\chi$, and $\phi(N)$ denotes Euler's function.
\end{thm}

We will also need the following lemma.

\begin{lem}\label{lem ceta}
Let $c$ denote complex conjugation on $Y_1(N)(\C)$. For any even non-trivial Dirichlet characters $\chi,\chi' : (\Z/N\Z)^\times \to \C^\times$, we have $c^* \eta(u_\chi,u_{\chi'}) = -\eta(u_\chi,u_{\chi'})$.
\end{lem}

\begin{proof}[Proof]
Recall that $c$ is given by $c(z)=-\overline{z}$ on $\h$. We have $c^* \log |u_\chi| = \log |u_\chi|$, and $c^*$ exchanges the holomorphic and anti-holomorphic parts of $\operatorname{dlog} |u_\chi|$. Since $\operatorname{darg}(u_\chi) = -i (\partial-\overline{\partial}) \log |u_\chi|$, we get $c^* \operatorname{darg}(u_\chi) = -\operatorname{darg}(u_\chi)$, and thus $c^* \eta(u_\chi,u_{\chi'}) = -\eta(u_\chi,u_{\chi'})$.
\end{proof}

\begin{remark}
By \cite[Prop. 5.4 and Prop. 6.1]{brunault:smf}, we have $\{u_\chi,u_{\chi'}\} \in K_2(X_1(N)(\C)) \otimes \C$. This implies that for $\gamma \in H_1(Y_1(N)(\C),\Z)$, the integral $\int_\gamma \eta(u_{\chi},u_{\chi'})$ depends only on the image of $\gamma$ in $H_1(X_1(N)(\C),\Z)$ (see for example the discussion in \cite[\S 3]{dokchitser-dejeu-zagier}). Therefore, we have a well-defined map
\begin{equation*}
\int \eta(u_{\chi},u_{\chi'}) : H_1(X_1(N)(\C),\Z) \to \C.
\end{equation*}
It can be extended by linearity to $H_1(X_1(N)(\C),\C)$.
\end{remark}

\begin{remark}\label{rem int eta}
Since $c^* \eta(u_\chi,u_{\chi'}) = -\eta(u_\chi,u_{\chi'})$ by Lemma \ref{lem ceta}, we have $\int_\gamma \eta(u_\chi,u_{\chi'}) = \int_{\gamma^-} \eta(u_\chi,u_{\chi'})$ with $\gamma^- = \frac12(\gamma-c_* \gamma)$.
\end{remark}

\section{Merel's formula}

In this section, we express the regulator integral appearing in the right hand side of (\ref{explicit beilinson formula}) as a linear combination of periods. In order to do this, we use an idea of Merel to express the integral over $X_1(N)(\C)$ as a linear combination of products of $1$-dimensional integrals.

Let $N \geq 1$ be an integer. Let $E_N$ be the set of vectors $(u,v) \in (\Z/N\Z)^2$ such that $(u,v,N)=1$. For any $f \in S_2(\Gamma_1(N))$ and any $x \in E_N$, we define the \emph{Manin symbol}
\begin{equation*}
\xi_f(x) = -\frac{1}{2\pi} \langle \xi(x),f \rangle = -i \int_{g_x 0}^{g_x \infty} f(z) dz,
\end{equation*}
where $g_x \in \SL_2(\Z)$ is any matrix whose bottom row is congruent to $x$ modulo $N$.

Let $\rho = e^{\frac{\pi i}{3}}$ and $\sigma = \begin{pmatrix} 0 & -1 \\ 1 & 0\end{pmatrix}$, $\tau= \begin{pmatrix} 0 & -1 \\ 1 & -1 \end{pmatrix}$, $T=\begin{pmatrix} 1 & 1 \\ 0 & 1 \end{pmatrix} \in \SL_2(\Z)$.

The following theorem is a variant of a theorem of Merel which expresses the Petersson scalar product of two cusp forms $f$ and $g$ of weight 2 as a linear combination of products of Manin symbols of $f$ and $g$ \cite[Théorème 2]{merel:symbmanin}.

\begin{thm}\label{thm reg eta}
Let $f \in S_2(\Gamma_1(N))$ be a cusp form of weight 2 and level $N$, and let $u,v \in \mathcal{O}^\times(Y_1(N)(\C))$ be two modular units. We have
\begin{equation}\label{eq reg eta}
\int_{X_1(N)(\C)} \eta(u,v) \wedge \omega_f = \frac{\pi}{2} \sum_{x \in E_N} \left(\int_{g_x \rho}^{g_x \rho^2} \eta(u,v) \right) \xi_f(x).
\end{equation}
\end{thm}

\begin{proof}[Proof]
Let $\mathcal{F}$ be the standard fundamental domain of $\SL_2(\Z) \backslash \h$:
\begin{equation*}
\mathcal{F} = \{z \in \h : |\mathrm{Re}(z)| \leq \frac12, |z| \geq 1\}.
\end{equation*}
Its boundary $\partial \mathcal{F}$ is the hyperbolic triangle with vertices $\rho^2,\rho,\infty$. Define
\begin{equation*}
F_x(z) = \int_\infty^z \omega_f | g_x \qquad (x \in E_N, z \in \h).
\end{equation*}
We have
\begin{equation*}
\int_{X_1(N)(\C)} \eta(u,v) \wedge \omega_f = \sum_{x \in E_N/\pm 1} \int_{\mathcal{F}} (\eta(u,v) \wedge \omega_f) | g_x.
\end{equation*}
Since $\eta(u,v)$ is closed, we have $(\eta(u,v) \wedge \omega_f) | g_x = -d(F_x \cdot (\eta(u,v)|g_x))$ and Stokes' formula gives
\begin{align}
\nonumber \int_{X_1(N)(\C)} \eta(u,v) \wedge \omega_f & = - \sum_{x \in E_N/\pm 1} \int_{\partial \mathcal{F}} F_x \cdot (\eta(u,v)|g_x)\\
\label{eq int} & = - \sum_{x \in E_N/\pm 1} \left(\int_{\rho^2}^\rho + \int_\rho^\infty + \int_\infty^{\rho^2}\right) F_x \cdot (\eta(u,v)|g_x).
\end{align}
The matrix $T$ fixes $\infty$ and maps $\rho^2$ to $\rho$. We have
\begin{equation*}
F_x(Tz) = \int_\infty^{Tz} \omega_f | g_x = \int_\infty^{z} \omega_f | g_x T = F_{xT}(z).
\end{equation*}
It follows that
\begin{align*}
\sum_{x \in E_N/\pm 1} \int_\rho^\infty F_x \cdot (\eta(u,v)|g_x) & = \sum_{x \in E_N/\pm 1} \int_{\rho^2}^\infty F_x | T \cdot (\eta(u,v)|g_x T)\\
& = \sum_{x \in E_N/\pm 1} \int_{\rho^2}^\infty F_{xT} \cdot (\eta(u,v)|g_{xT})\\
& = \sum_{x \in E_N/\pm 1} \int_{\rho^2}^\infty F_x \cdot (\eta(u,v)|g_x).
\end{align*}
Hence (\ref{eq int}) simplifies to
\begin{equation*}
\int_{X_1(N)(\C)} \eta(u,v) \wedge \omega_f = \sum_{x \in E_N/\pm 1} \int_{\rho}^{\rho^2} F_x \cdot (\eta(u,v)|g_x).
\end{equation*}
Similarly, let us use the matrix $\sigma$, which exchanges $\rho$ and $\rho^2$, as well as $0$ and $\infty$. Since $F_x(\sigma z) = F_{x \sigma}(z)+2\pi \xi_f(x)$, we get
\begin{equation*}
\int_{\rho}^{\rho^2} F_x \cdot (\eta(u,v)|g_x) = \int_{\rho^2}^{\rho} F_{x\sigma} \cdot (\eta(u,v)|g_{x\sigma}) + 2\pi \xi_f(x) \int_{\rho^2}^{\rho} \eta(u,v) | g_x.
\end{equation*}
Summing over $x$ and using the fact that $\xi_f(x\sigma)=-\xi_f(x)$, we get
\begin{align*}
\int_{X_1(N)(\C)} \eta(u,v) \wedge \omega_f & = \frac12 \sum_{x \in E_N/\pm 1} 2\pi \xi_f(x) \int_{\rho^2}^{\rho} \eta(u,v) | g_{x\sigma}\\
& = \pi \sum_{x \in E_N/\pm 1} \xi_f(x) \int_\rho^{\rho^2} \eta(u,v) | g_{x}.
\end{align*}
\end{proof}

\begin{remark}
It can be shown that if $\{u,v\}$ defines an element in $K_2(X_1(N)(\C)) \otimes \Q$, then the cycle $\sum_{x \in E_N} \left(\int_{g_x \rho}^{g_x \rho^2} \eta(u,v)\right) \xi(x)$ is \emph{closed}. This follows from the fact that if $\gamma_P$ denotes a small loop around a cusp $P$ of $X_1(N)(\C)$, then $\int_{\gamma_P} \eta(u,v) = 2\pi  \log |\partial_P(u,v)|$, where $\partial_P(u,v)$ denotes the tame symbol of $\{u,v\}$ at $P$ (see for example \cite[\S 4, Lemma]{rodriguez:modular}).
\end{remark}

\begin{definition}
Let $f \in S_2(\Gamma_1(N))$ be a cusp form of weight 2 and level $N$. Consider the following relative cycle on $Y_1(N)(\C)$:
\begin{equation*}
\gamma_f := \sum_{x \in E_N} \xi_f(x) \{g_x \rho, g_x \rho^2\}.
\end{equation*}
Furthermore, let us define $\gamma_f^- := \frac12 (\gamma_f-c_* \gamma_f)$.
\end{definition}

Combining Theorem \ref{explicit beilinson}, Theorem \ref{thm reg eta} and Remark \ref{rem int eta}, we get the following result.

\begin{thm}\label{explicit beilinson 2}
Let $f \in S_2(\Gamma_1(N),\psi)$ be a newform of weight $2$, level $N$ and character $\psi$. For any even primitive Dirichlet character $\chi : (\Z/N\Z)^\times \to \C^\times$, with $\chi \neq \overline{\psi}$, we have
\begin{equation}\label{explicit beilinson formula 2}
L(f,2) L(f,\chi,1) = \frac{N \pi^2 \tau(\chi)}{4 \phi(N)} \int_{\gamma_f} \eta(u_{\overline{\chi}},u_{\psi\chi}) = \frac{N \pi^2 \tau(\chi)}{4 \phi(N)} \int_{\gamma_f^-} \eta(u_{\overline{\chi}},u_{\psi\chi}).
\end{equation}
\end{thm}

We will also need an explicit expression of $\gamma_f$ in terms of Manin symbols. For any $f \in S_2(\Gamma_1(N))$ and any $x=(u,v) \in E_N$, let us define $x^c=(-u,v)$ and
\begin{equation*}
\xi_f^+(x) = \frac12 (\xi_f(x)+\xi_f(x^c)) = \frac12 (\xi_f(x)+\overline{\xi_{f^*}(x)}),
\end{equation*}
where $f^*$ denotes the cusp form with complex conjugate Fourier coefficients.

\begin{pro}\label{pro gammaf}
Let $f \in S_2(\Gamma_1(N))$ be a cusp form of weight $2$ and level $N$. The cycle $\gamma_f$ is closed, and its image in $H_1(X_1(N)(\C),\Z)$ can be expressed as follows:
\begin{equation}\label{eq gammaf}
\gamma_f = -\frac13 \sum_{x \in E_N} \left(\xi_f(x)+2\xi_f(x \tau)\right) \xi(x).
\end{equation}
Moreover, we have
\begin{equation}\label{eq gammaf 2}
\gamma_f^- = -\frac13 \sum_{x \in E_N} \left(\xi_f^+(x)+2\xi_f^+(x \tau)\right) \xi(x).
\end{equation}
\end{pro}

\begin{proof}[Proof]
Let us compute the boundary of $\gamma_f$. Since $\sigma(\rho)=\rho^2$ and $\xi_f(x\sigma)=-\xi_f(x)$, we have
\begin{align*}
\partial \gamma_f & = \sum_{x \in E_N} \xi_f(x) ([g_x \rho^2]-[g_x \rho])\\
& = \sum_{x \in E_N} \xi_f(x) ([g_{x\sigma} \rho]-[g_x \rho])\\
& = -2 \sum_{x \in E_N} \xi_f(x) [g_x \rho].
\end{align*}
Since $\tau(\rho)=\rho$ and because of Manin's relation $\xi_f(x)+\xi_f(x\tau)+\xi_f(x\tau^2)=0$, we get
\begin{align*}
\partial \gamma_f & = -\frac23 \sum_{x \in E_N} \xi_f(x) ([g_x \rho] + [g_{x \tau} \rho]+[g_{x \tau^2} \rho])\\
& = -\frac23 \sum_{x \in E_N} (\xi_f(x)+\xi_f(x \tau)+\xi_f(x \tau^2)) [g_x \rho] = 0.
\end{align*}
On the other hand, we have
\begin{align*}
\gamma_f & = \sum_{x \in E_N} \xi_f(x) (\{g_x \rho, g_x \infty\}+ \{g_x \infty, g_x \rho^2\})\\
& = \sum_{x \in E_N} \xi_f(x) (\{g_x \rho, g_x \infty\} - \sum_{x \in E_N} \xi_f(x) \{g_x 0, g_x \rho\})\\
& = 2 \sum_{x \in E_N} \xi_f(x) \{g_x \rho, g_x \infty\} - \sum_{x \in E_N} \xi_f(x) \xi(x).
\end{align*}
Using the matrix $\tau$, we get
\begin{align*}
\gamma_f & = \frac23 \sum_{x \in E_N} \left(\xi_f(x) \{g_x \rho, g_x \infty\}+\xi_f(x\tau) \{g_{x\tau} \rho, g_{x\tau} \infty\}+\xi_f(x\tau^2) \{g_{x\tau^2} \rho, g_{x\tau^2} \infty\}\right) - \sum_{x \in E_N} \xi_f(x) \xi(x)\\
& = \frac23 \sum_{x \in E_N} \left(\xi_f(x) \{g_x \rho, g_x \infty\}+\xi_f(x\tau) \{g_{x} \rho, g_{x} 0\}+\xi_f(x\tau^2) \{g_{x} \rho, g_{x} 1\}\right) - \sum_{x \in E_N} \xi_f(x) \xi(x)\\
& = \frac23 \sum_{x \in E_N} \left(\xi_f(x\tau) \{g_{x} \infty, g_{x} 0\}+\xi_f(x\tau^2) \{g_{x} \infty, g_{x} 1\}\right) - \sum_{x \in E_N} \xi_f(x) \xi(x)\\
& = \frac23 \sum_{x \in E_N} \left(-\xi_f(x\tau) \xi(x) +\xi_f(x\tau^2) \{g_{x\tau^2} 0, g_{x\tau^2} \infty\}\right) - \sum_{x \in E_N} \xi_f(x) \xi(x)\\
& = \frac23 \sum_{x \in E_N} \left(-\xi_f(x\tau) \xi(x) +\xi_f(x) \xi(x)\right) - \sum_{x \in E_N} \xi_f(x) \xi(x)\\
& = \frac13 \sum_{x \in E_N} (\xi_f(x) - 2 \xi_f(x\tau)) \xi(x).
\end{align*}
This gives (\ref{eq gammaf}). The action of complex conjugation on $\gamma_f$ is given by
\begin{align*}
c_* \gamma_f & = \sum_{x \in E_N} \xi_f(x) \{c(g_x \rho),c(g_x \rho^2)\}\\
& = \sum_{x \in E_N} \xi_f(x) \{g_{x^c} \rho^2,g_{x^c} \rho\}\\
& = - \sum_{x \in E_N} \xi_f(x^c) \{g_x \rho,g_x \rho^2\}.
\end{align*}
It follows that
\begin{equation*}
\gamma_f^- = \sum_{x \in E_N} \xi_f^+(x) \{g_x \rho,g_x \rho^2\}.
\end{equation*}
Since the quantities $\xi_f^+(x)$ satisfy the Manin relations, the same proof as above gives (\ref{eq gammaf 2}).
\end{proof}

\section{Proof of the main theorem}

Let us return to the case $N=13$. Using Theorem \ref{explicit beilinson 2} with $f=f_\eps$, $\psi=\eps$ and $\chi=\eps^3$, we get
\begin{equation}\label{formula 1}
L(f_\eps,2) L(f_\eps,\eps^3,1) = \frac{13 \pi^2 \tau(\eps^3)}{48} \int_{\gamma_{f_\eps}^-} \eta(u_{\eps^3},u_{\epsb^2}).
\end{equation}
We are going to make explicit each term in this formula. Note that $\tau(\eps^3)=\sqrt{13}$.

\begin{definition}
For any Dirichlet character $\psi : (\Z/13\Z)^\times \to \C^\times$, let us denote $\mathcal{H}(\psi)$ (resp. $\hat{\mathcal{H}}(\psi)$) the $\psi$-isotypical component of $\mathcal{H} \otimes \C$ (resp. $\hat{\mathcal{H}} \otimes \C$) with respect to the action of diamond operators $\langle d \rangle_*$, $d \in (\Z/13\Z)^\times$. For any $\gamma \in \hat{\mathcal{H}} \otimes \C$, let $\gamma^\psi$ denote its $\psi$-isotypical component. Moreover, let us define $\hat{\mathcal{H}}^{\pm}(\psi) = (\hat{\mathcal{H}}^\pm \otimes \C) \cap \hat{\mathcal{H}}(\psi)$ and $\mathcal{H}^{\pm}(\psi) = (\mathcal{H}^\pm \otimes \C) \cap \mathcal{H}(\psi)$.
\end{definition}

\begin{lem}\label{lem Hpsi}
Let $\psi=\varepsilon$ or $\overline{\varepsilon}$. Then $\mathcal{H}^{\pm}(\psi)$ has dimension $1$, and a generator is given by
\begin{align*}
\gamma_\psi^+ & := \sum_{a \in (\Z/13\Z)^\times} \varepsilon^3(a) \xi(1,a)^\psi\\
\gamma_\psi^- & := \xi(1,-3)^\psi - \xi(1,3)^\psi.
\end{align*}
Moreover, we have $W_{13} \gamma_\psi^+ = \psi(2) \gamma_{\overline{\psi}}^+$.
\end{lem}

\begin{proof}[Proof]
The pairing $\langle \cdot,\cdot \rangle$ induces a perfect pairing
\begin{equation*}
\mathcal{H}^{\pm}(\psi) \times S_2(\Gamma_1(13),\psi) \to \C.
\end{equation*}
Since $S_2(\Gamma_1(13),\psi)$ is $1$-dimensional, we get $\dim_\C \mathcal{H}^{\pm}(\psi)=1$. From the definition, it is clear that $\gamma_\psi^+ \in \hat{\mathcal{H}}^+(\psi)$ and $\gamma_\psi^- \in \hat{\mathcal{H}}^-(\psi)$. Moreover, since $\gamma_\psi^- = \gamma_3^\psi$, we have $\gamma_\psi^- \in \mathcal{H}^-(\psi)$.

Let us compute the boundary of $\gamma_\psi^+$. For any $u,v \in (\Z/13\Z)^\times$, we have $\partial \xi(u,v) = P_u - P_v$ with $P_d:=\langle d \rangle(0)$. Moreover, for any $x \in E_{13}$, we have
\begin{equation*}
\xi(x)^\psi = \frac{1}{12} \sum_{d \in (\Z/13\Z)^\times} \overline{\psi}(d) \langle d \rangle_* \xi(x) = \frac{1}{12} \sum_{d \in (\Z/13\Z)^\times} \overline{\psi}(d) \xi(d x).
\end{equation*}
It follows that
\begin{align*}
\partial \gamma_\psi^+ & = \sum_{a \in (\Z/13\Z)^\times} \varepsilon^3(a) \partial(\xi(1,a)^\psi)\\
& = \frac{1}{12} \sum_{a \in (\Z/13\Z)^\times} \varepsilon^3(a) \sum_{d \in (\Z/13\Z)^\times} \overline{\psi}(d) \partial \xi(d,da)\\
& = \frac{1}{12} \sum_{a \in (\Z/13\Z)^\times} \varepsilon^3(a) \sum_{d \in (\Z/13\Z)^\times} \overline{\psi}(d) (P_d-P_{da})\\
& = \frac{1}{12} \sum_{d \in (\Z/13\Z)^\times} \left(\sum_{a \in (\Z/13\Z)^\times} \varepsilon^3(a)-\varepsilon^3\psi(a)\right)  \overline{\psi}(d) \cdot P_d = 0.
\end{align*}
Hence $\gamma_\psi^+ \in \mathcal{H}^+(\psi)$. By \cite[Lemme 5]{rebolledo}, the elements $\xi(1,0)^\psi, \xi(1,2)^\psi, \xi(1,3)^\psi, \xi(1,-3)^\psi$ form a basis of $\hat{\mathcal{H}}(\psi)$, and we can express $\gamma_\psi^+$ in terms of this basis. This gives
\begin{equation}\label{eq gammapsi}
\gamma_\psi^+ = (2-4\psi(2)) \xi(1,2)^\psi + \xi(1,3)^\psi + \xi(1,-3)^\psi.
\end{equation}
In particular $\gamma_\psi^+$ and $\gamma_\psi^-$ are nonzero, and thus they generate $\mathcal{H}^{\pm}(\psi)$.

It remains to compute the action of $W_{13}$ on $\gamma_\psi^+$. In view of (\ref{eq gammapsi}), it is enough to determine the action of $W_{13}$ on $\xi(1,2)$ and $\xi(1,3)$. We have
\begin{align*}
W_{13} \xi(1,2) & = \left\{\frac{2}{13},\infty\right\} = \left\{\frac{2}{13},\frac{1}{6}\right\}+\left\{\frac{1}{6},0\right\}+\{0,\infty\}\\
& = -\xi(0,-6)+\xi(1,-6)+\xi(0,1).
\end{align*}
Hence, using \cite[Lemme 5]{rebolledo} again, we get
\begin{align*}
W_{13} (\xi(1,2)^\psi) & = -\xi(0,-6)^{\overline{\psi}}+\xi(1,-6)^{\overline{\psi}}+\xi(0,1)^{\overline{\psi}}\\
& = (\overline{\psi}(6)-1) \xi(1,0)^{\overline{\psi}} - \overline{\psi}(6) \xi(1,2)^{\overline{\psi}}.
\end{align*}
Similarly, we find
\begin{align*}
W_{13} (\xi(1,3)^\psi) & = (\overline{\psi}(4)-1) \xi(1,0)^{\overline{\psi}} - \overline{\psi}(4) \xi(1,-3)^{\overline{\psi}}\\
W_{13} (\xi(1,-3)^\psi) & = (\overline{\psi}(4)-1) \xi(1,0)^{\overline{\psi}} - \overline{\psi}(4) \xi(1,3)^{\overline{\psi}}.
\end{align*}
Since we know that $W_{13} \gamma_\psi^+$ is a multiple of $\gamma_{\overline{\psi}}^+$, we deduce $W_{13} \gamma_\psi^+ = -\overline{\psi}(4) \gamma_{\overline{\psi}}^+ = \psi(2) \gamma_{\overline{\psi}}^+$.
\end{proof}

\begin{pro}\label{pro Lfeps3}
We have $L(f_\varepsilon,\varepsilon^3,1)=\frac{\overline{\varepsilon}(2)}{\sqrt{13}} \langle \gamma_\varepsilon^+, f_\varepsilon \rangle$.
\end{pro}

\begin{proof}[Proof]
By \cite[Thm 4.2.b)]{manin}, we have
\begin{equation*}
L(f_\varepsilon,\varepsilon^3,1) = \frac{1}{\sqrt{13}} \sum_{a \in (\Z/13\Z)^\times} \varepsilon^3(a) \int_{a/13}^{\infty} \omega_{f_\varepsilon}.
\end{equation*}
Let us compute the cycle $\theta = \sum_{a \in (\Z/13\Z)^\times} \varepsilon^3(a) \{\frac{a}{13},\infty\}$ in terms of Manin symbols. We have
\begin{equation*}
W_{13} (\theta^\varepsilon) = (W_{13} \theta)^{\overline{\varepsilon}} = \sum_{a \in (\Z/13\Z)^\times} \varepsilon^3(a) \left\{-\frac{1}{a},0\right\}^{\overline{\varepsilon}} = \sum_{a \in (\Z/13\Z)^\times} \varepsilon^3(a) \xi(1,a)^{\overline{\varepsilon}} = \gamma_{\overline{\varepsilon}}^+.
\end{equation*}
By Lemma \ref{lem Hpsi}, it follows that
\begin{equation*}
\langle \theta, f_\varepsilon \rangle = \langle \theta^\varepsilon, f_\varepsilon \rangle = \langle W_{13}(\gamma_{\overline{\varepsilon}}^+), f_\varepsilon \rangle = \overline{\varepsilon}(2) \langle \gamma_\varepsilon^+, f_\varepsilon \rangle.
\end{equation*}
\end{proof}

\begin{pro}\label{pro gammafeps}
We have $\gamma_{f_\varepsilon}^- = \frac{1-2\zeta_6}{\pi} \langle \gamma_\varepsilon^+, f_\varepsilon \rangle \cdot \gamma_{\overline{\varepsilon}}^-$.
\end{pro}

\begin{proof}[Proof]
By Proposition \ref{pro gammaf}, we have
\begin{equation*}
\gamma_{f_\varepsilon}^- = - \frac13 \sum_{x \in E_{13}} (\xi_{f_\eps}^+(x)+2\xi_{f_\eps}^+(x\tau)) \xi(x).
\end{equation*}
This sum involves 168 terms, but we may reduce it to 14 terms by considering the action of diamond operators. Let $\mathcal{E}$ be the set of 2-tuples $(0,1)$ and $(1,v)$, $v \in \Z/13\Z$. We have
\begin{align*}
\gamma_{f_\varepsilon}^- & = - \frac13 \sum_{x \in \mathcal{E}} \sum_{d \in (\Z/13\Z)^\times} (\xi_{f_\eps}^+(dx)+2\xi_{f_\eps}^+(dx\tau)) \xi(dx)\\
& = - \frac13 \sum_{x \in \mathcal{E}} \sum_{d \in (\Z/13\Z)^\times} (\xi_{f_\eps}^+(x)+2\xi_{f_\eps}^+(x\tau)) \cdot \eps(d) \langle d \rangle_* \xi(x)\\
& = -4 \sum_{x \in \mathcal{E}} (\xi_{f_\eps}^+(x)+2\xi_{f_\eps}^+(x\tau)) \xi(x)^\epsb.
\end{align*}
A simple computation shows that the terms $x=(0,1)$ and $x=(1,0)$ cancel each other. Hence
\begin{equation*}
\gamma_{f_\varepsilon}^- = -4 \sum_{v \in (\Z/13\Z)^{*}} \bigl(\xi^{+}_{f_{\eps}}(1,v)+2 \eps(v) \xi^{+}_{f_{\eps}}(1,1+\frac{1}{v})\bigr) \cdot \xi(1,v)^{\epsb}.
\end{equation*}
Using \cite[Lemme 5]{rebolledo}, we may express $\xi_{f_\eps}^+(1,v)$, $v \neq 0$ in terms of $\xi_{f_\eps}^+(1,2)$ and $\xi_{f_\eps}^+(1,3)$. We find $\xi^{+}_{f_{\eps}}(1,-v) = \xi^{+}_{f_{\eps}}(1,v)$ and
\begin{align*}
\xi^{+}_{f_{\eps}}(1,1) & = 0 & \xi^{+}_{f_{\eps}}(1,4) & = (1-\zeta_{6})
\xi^{+}_{f_{\eps}}(1,3)\\
\xi^{+}_{f_{\eps}}(1,5) & = (\zeta_{6}-1) \bigl(\xi^{+}_{f_{\eps}}(1,2) -
\xi^{+}_{f_{\eps}}(1,3)\bigr) &  \xi^{+}_{f_{\eps}}(1,6)& = (\zeta_{6}-1) \xi^{+}_{f_{\eps}}(1,2).
\end{align*}
Moreover, also by \cite[Lemme 5]{rebolledo}, the cycles $\xi(1,v)^\epsb$, $v \neq 0$, are linear combinations of $\xi(1,2)^\epsb$, $\xi(1,3)^\epsb$ and $\xi(1,-3)^\epsb$. Thus the same is true for $\gamma_{f_\varepsilon}^-$. But we know that $\gamma_{f_\varepsilon}^-$ is a multiple of $\gamma^{-}_{\epsb} = \xi(1,3)^{\epsb}-\xi(1,-3)^{\epsb}$. It is thus enough to compute the coefficient in front of $\xi(1,3)^{\epsb}$, which leads to the identity
\begin{equation*}
\gamma_{f_\varepsilon}^- = \left(12\xi^{+}_{f_{\eps}}(1,2)+(8\zeta_{6}-4) \xi^{+}_{f_{\eps}}(1,3)\right) \cdot \gamma^-_\epsb.
\end{equation*}
Using (\ref{eq gammapsi}) with $\psi=\eps$, we get the proposition.
\end{proof}

Consider the modular units $x=W_{13}(h)$ and $y=W_{13}(H)$.

\begin{pro}\label{pro etaxy}
We have $\int_{\gamma_\epsb^-} \eta(x,y) = \frac{13^2 \sqrt{13}}{48}(1+\zeta_6) \tau(\eps^2) \int_{\gamma_\epsb^-} \eta(u_{\eps^3},u_{\epsb^2})$.
\end{pro}

\begin{proof}[Proof]
Since $h$ and $H$ are supported in the cusps above $0 \in X_0(13)(\Q)$, it follows that $x$ and $y$ are supported in the cusps above $\infty \in X_0(13)(\Q)$, namely the cusps $\langle d \rangle \infty$, $d \in (\Z/13\Z)^\times/\pm 1$. The method of proof is simple : we decompose the divisors of $x$ and $y$ as linear combinations of Dirichlet characters.

Let us write $\begin{pmatrix} n_1 & n_2 & \cdots & n_6 \end{pmatrix}$ for the divisor $\sum_{d=1}^6 n_d \cdot \langle d \rangle \infty$. By \cite[p. 56]{lecacheux:13}, we have
\begin{align*}
\dv(x) & = \begin{pmatrix} 0 & 1 & 1 & -1 & 0 & -1 \end{pmatrix}\\
\dv(y) & = \begin{pmatrix} 1 & -1 & 1 & 1 & -1 & -1 \end{pmatrix}.
\end{align*}
The divisors of $u_{\eps^3}$ and $u_{\epsb^2}$ are given by \cite[Prop 5.4]{brunault:smf}. We have
\begin{equation*}
\dv (u_{\eps^{3}}) = -\frac{L(\eps^{3},2)}{\pi^{2}} \cdot \begin{pmatrix} 1 & -1 & 1 & 1 & -1 & -1 \end{pmatrix} = -\frac{4 \sqrt{13}}{13^{2}} \dv (y).
\end{equation*}
Since the divisor of $x$ is invariant under the diamond operator $\langle 5 \rangle$, it is a linear combination of $\dv(u_{\eps^2})$ and $\dv(u_{\epsb^2})$. We find explicitly
\begin{align*}
\dv(x) & = \frac{1-2\zeta_{6}}{3} \Bigl(\frac{\dv(u_{\eps^{2}})}{L(\eps^{2},2)/\pi^{2}} - \frac{\dv(u_{\epsb^{2}})}{L(\epsb^{2},2)/\pi^{2}}\Bigr)\\
& = \frac{13}{12} \bigl((2-\zeta_{6}) \tau(\epsb^{2}) \dv(u_{\eps^{2}}) + (1+\zeta_{6})
\tau(\eps^{2}) \dv(u_{\epsb^{2}})\bigr).
\end{align*}
Here we have used the classical formula \cite[(1.80) and (3.87)]{cartier}
\begin{equation*}
\frac{L(\chi,2)}{\pi^2} = \frac{\tau(\chi)}{N} \sum_{a=0}^{N-1} \chib(a) B_2\left(\frac{a}{N}\right)
\end{equation*}
where $\chi$ is an even non-trivial Dirichlet character modulo $N$, and $B_2(x)=x^2-x+\frac16$ is the second Bernoulli polynomial.

Considering $u_{\eps^3}$ and $u_{\epsb^2}$ as elements of $\mathcal{O}^{*}(Y_{1}(13)(\C)) \otimes
\C$ and following the notations of \cite[(65)]{brunault:smf}, we have $\widehat{u_{\eps^3}}(\infty)=\widehat{u_{\epsb^2}}(\infty)=1$ by \cite[Prop. 5.3]{brunault:smf}. Moreover, looking at the behaviour of $x$ and $y$ at $\infty$, we find $x(\infty)=1$ and $\widehat{y}(\infty)=-1$. Hence $x \otimes 1$ can be expressed as a linear combination of $u_{\eps^2}$ and $u_{\epsb^2}$ in $\mathcal{O}^{*}(Y_{1}(13)(\C)) \otimes
\C$, while $y \otimes 1$ is proportional to $u_{\eps^3}$. Thus
\begin{equation*}
\eta(x,y) = -\frac{13^2}{4 \sqrt{13}} \cdot \frac{13}{12} \left((2-\zeta_{6}) \tau(\epsb^{2}) \eta(u_{\eps^{2}},u_{\eps^3}) + (1+\zeta_{6}) \tau(\eps^{2}) \eta(u_{\epsb^{2}},u_{\eps^3})\right).
\end{equation*}
Since the differential form $\eta(u_{\eps^2},u_{\eps^3})$ has character $\eps$, we have $\int_{\gamma_{\epsb}^-} \eta(u_{\eps^2},u_{\eps^3}) = 0$, and the proposition follows.
\end{proof}

\begin{proof}[Proof of Theorem \ref{main thm}]
Combining (\ref{formula 1}) with Propositions \ref{pro Lfeps3}, \ref{pro gammafeps}, \ref{pro etaxy}, we get
\begin{equation}\label{formula 2}
L(f_\eps,2) = \frac{\pi}{\sqrt{13}} \cdot \frac{1-\zeta_6}{\tau(\eps^2)} \int_{\gamma_{\epsb}^-} \eta(x,y).
\end{equation}
Formula (\ref{formula 2}) simplifies if we use the functional equation of $L(f_\eps,s)$. Recall that $W_{13}(f_\eps) = w f_{\epsb}$. Let $\Lambda(f,s):=13^{s/2} (2\pi)^{-s} \Gamma(s) L(f,s)$. Then the functional equation of $L(f_\eps,s)$ reads
\begin{equation*}
\Lambda(f_\eps,s) = -w \Lambda(f_{\epsb},2-s).
\end{equation*}
Using (\ref{eq w}), we deduce that
\begin{equation*}
L(f_\eps,2) = \frac{4\pi^2}{13^2} (4-3\zeta_6) \tau(\eps) L'(f_\epsb,0).
\end{equation*}
Replacing in (\ref{formula 2}) and using $\tau(\eps^2) \tau(\eps) = (4\zeta_6-3) \sqrt{13}$, we get
\begin{equation}\label{formula 3}
\int_{\gamma_\epsb^-} \eta(x,y) = 4\pi (\zeta_6-1) L'(f_\epsb,0).
\end{equation}
Taking complex conjugation, and since $\overline{\eta(x,y)}=\eta(x,y)$, we obtain
\begin{equation}\label{formula 4}
\int_{\gamma_\eps^-} \eta(x,y) = -4\pi \zeta_6 L'(f_\eps,0).
\end{equation}
We have a direct sum decomposition $\mathcal{H}^- \otimes \C = \mathcal{H}^-(\eps) \oplus \mathcal{H}^-(\epsb)$. Write $\gamma_3 = \gamma_3^\eps + \gamma_3^{\epsb}$. Then $\gamma_4 = \langle 2 \rangle_* \gamma_3 = \eps(2) \gamma_3^\eps + \epsb(2) \gamma_3^\epsb$. By Proposition \ref{pro W13gamma0}, we deduce
\begin{equation*}
W_{13} \gamma_P = \gamma_4 - \gamma_3 = (\zeta_6-1) \gamma_3^\eps + (\overline{\zeta_6}-1) \gamma_3^\epsb = (\zeta_6-1) \gamma_\eps^- + (\overline{\zeta_6}-1) \gamma_\epsb^-.
\end{equation*}
By (\ref{formula 3}) and (\ref{formula 4}), we then have
\begin{align*}
\int_{W_{13} \gamma_P} \eta(x,y) & = (\zeta_6-1) \int_{\gamma_\eps^-} \eta(x,y) + (\overline{\zeta_6}-1) \int_{\gamma_\epsb^-} \eta(x,y)\\
& = 4\pi (L'(f_\eps,0) + L'(f_\epsb,0)).
\end{align*}
By Proposition \ref{pro deninger}, we conclude that
\begin{equation*}
m(P)=\frac{1}{2\pi} \int_{\gamma_P} \eta(h,H) = \frac{1}{2\pi} \int_{W_{13} \gamma_P} \eta(x,y) = 2 L'(f,0).
\end{equation*}
\end{proof}

\begin{remark}
There may have been a quicker way to proceed. Starting from Theorem \ref{explicit beilinson} in the particular case $N=13$, probably all we need is a symplectic basis of $H_1(X_1(13)(\C),\Z)$ with respect to the intersection pairing (see the formula \cite[A.2.5]{bost}). But this is less canonical than Theorem \ref{thm reg eta}.
\end{remark}

\begin{remark}
Another way of proving Theorem \ref{main thm} would be to use the main formula of \cite{zudilin}. We have not worked out the details of this computation.
\end{remark}

\begin{question}
Let $g = f | \langle 2 \rangle = \zeta_6 f_\eps+ \overline{\zeta_6} f_\epsb$. Then $(f,g)$ is a basis of the space $S_2(\Gamma_1(13),\Q)$ of cusp forms with rational Fourier coefficients. Is there a polynomial $Q \in \Z[x,y]$ such that $m(Q)$ is proportional to $L'(g,0)$?
\end{question}

\section{Examples in higher level}

We note that the functions $H$ and $h$ used in the proof of Theorem \ref{main thm} are modular units on $X_1(13)$ and that $P$ is their minimal polynomial. There is a similar story for the modular curve $X_1(11)$ \cite[Cor 3.3]{brunault:cras} and we may try to generalize this phenomenon.

Let $N \geq 1$ be an integer, and let $u$ and $v$ be two modular units on $X_1(N)$. Let $P \in \C[x,y]$ be an irreducible polynomial such that $P(u,v)=0$. Then the map $z \mapsto (u(z),v(z))$ is a modular parametrization of the curve $C_P : P(x,y)=0$ and we have a natural map $Y_1(N) \to C_P$. Assuming $P$ satisfies Deninger's conditions, we may express $m(P)$ in terms of the integral of $\eta(u,v)$ over a (non necessarily closed) cycle $\gamma_P$.

The most favourable case is when the curve $C_P$ intersects the torus $T^2 = \{|x|=|y|=1\}$ only at cusps. In this case $\gamma_P$ is a modular symbol and we may use \cite{zudilin} to compute $\int_{\gamma_P} \eta(u,v)$ in terms of special values of $L$-functions.

In this section, we work out this idea for some examples of increasing complexity. We work with the modular units provided by \cite{yang}. These modular units are supported on the cusps above $\infty \in X_0(N)$, so that \cite[Prop 6.1]{brunault:smf} implies that $P$ is automatically tempered.

In all examples below, we found that $\gamma_P$ can be written as the sum of a closed path $\gamma_0$ and a path $\gamma_1$ joining cusps. The integral of $\eta(u,v)$ over $\gamma_1$ can be computed using \cite[Thm 1]{zudilin}. The integral of $\eta(u,v)$ over $\gamma_0$ can be dealt with using either \cite[Thm 1]{zudilin} or the explicit version of Beilinson's theorem -- we have not carried out the details of the computation. So in order to establish the identities below rigorously, it only remains to express $\gamma_0$ in terms of modular symbols and to compute $\int_{\gamma_0} \eta(u,v)$ using the tools explained above.

It would be interesting to understand when the identities obtained involve cusp forms (like (\ref{eq mP16})), are of Dirichlet type (like (\ref{eq mP18})), or of mixed type (like (\ref{eq mP25})). In the general case, it would be also interesting to find conditions on the modular units $u$ and $v$ so that the boundary of $\gamma_P$ consists of cusps or other interesting points.

\subsection{$N=16$}

The modular curve $X_1(16)$ has genus 2 and has been studied in \cite{lecacheux:16}. Let $u$ and $v$ be the following modular units:
\begin{align*}
u & = q \prod_{\substack{n \geq 1\\ n \equiv \pm 1,\pm 5 (16)}} (1-q^n) / \prod_{\substack{n \geq 1 \\ n \equiv \pm 3,\pm 7 (16)}} (1-q^n)\\
v & = q \prod_{\substack{n \geq 1\\ n \equiv \pm 14 (16)}} (1-q^n) / \prod_{\substack{n \geq 1 \\ n \equiv \pm 10(16)}} (1-q^n).
\end{align*}
Their minimal polynomial is given by
\begin{equation*}
P_{16} = y-x-xy-xy^2+x^2y+xy^3.
\end{equation*}
This polynomial vanishes on the torus at the points $(x,y)=(1,1)$, $(1,\pm i)$, $(-1,-1)$, but the Deninger cycle $\gamma_{P_{16}}$ is \emph{closed}. So we may expect that $m(P_{16})$ is equal to $L'(f,0)$ for some cusp form $f$ of level $16$ with rational coefficients. Indeed, we find numerically
\begin{equation}\label{eq mP16}
m(P_{16}) \stackrel{?}{=} L'(f,0)
\end{equation}
where $f$ is the trace of the unique newform of weight $2$ and level $16$, having coefficients in $\Z[i]$.

\subsection{$N=18$}

The modular curve $X_1(18)$ has genus $2$ and has been studied in \cite{lecacheux:18}. It has 3 cusps above $\infty$, so we may form essentially two modular units supported on these cusps. Let $u$ and $v$ be the following modular units:
\begin{align*}
u & = q^3 \prod_{\substack{n \geq 1\\ n \equiv \pm 1,\pm 2 (18)}} (1-q^n) / \prod_{\substack{n \geq 1 \\ n \equiv \pm 7,\pm 8 (18)}} (1-q^n)\\
v & = q^2 \prod_{\substack{n \geq 1\\ n \equiv \pm 1,\pm 4 (18)}} (1-q^n) / \prod_{\substack{n \geq 1 \\ n \equiv \pm 5, \pm 8 (18)}} (1-q^n).
\end{align*}
Their minimal polynomial is given by
\begin{equation*}
P_{18} = -x^2+y^3+xy^2-x^2y+x^2y^2-x^3y^2.
\end{equation*}
This polynomial vanishes on the torus at the points $(x,y)=(1,\pm 1)$, $(-1,\pm 1)$, $(\zeta_6^2,\zeta_6)$ and $(\overline{\zeta_6}^2,\overline{\zeta_6})$ with $\zeta_6=e^{2\pi i/6}$. The points $(\zeta_6^2,\zeta_6)$ and $(\overline{\zeta_6}^2,\overline{\zeta_6})$ correspond respectively to the cusps $\frac16$ and $-\frac16$, and the Deninger cycle $\gamma_{P_{18}}$ is given by $\gamma_0+\{-\frac16,\frac16\}$, where $\gamma_0$ is a closed cycle. Using \cite[Thm 1]{zudilin}, we find
\begin{equation*}
\int_{-1/6}^{1/6} \eta(u,v) = \frac{1}{4\pi} L(F,2)
\end{equation*}
where $F$ is a modular form of weight 2 and level (at most) $18^2$. Actually $F$ has level $18$ and \cite[Thm 1]{zudilin} simplifies if we use the functional equation $L(F,2)=-\frac{2\pi^2}{9} L'(W_{18} F,0)$. In fact \cite[Lemma 2]{zudilin} guarantees that $W_{18} F$ will be a modular form with \emph{integral} Fourier coefficients. In this case, we find
\begin{equation*}
W_{18} F = -36 E_2^{\psi}
\end{equation*}
where $E_2^{\psi} = \sum_{n = 1}^\infty (\sum_{d | n} d) \psi(n) q^n$ is an Eisenstein series of level $9$, and $\psi : (\Z/3\Z)^\times \to \{\pm 1\}$ is the unique Dirichlet character of conductor $3$. Since $L(E_2^{\psi},s)=L(\psi,s) L(\psi,s-1)$, we may expect that $m(P_{18})$ involves $L$-values of Dirichlet characters. Indeed, we find numerically
\begin{equation}\label{eq mP18}
m(P_{18}) \stackrel{?}{=} 2 L'(\psi,-1).
\end{equation}

\subsection{$N=25$}

The modular curve $X_1(25)$ has genus 12 and the quotient $X=X_1(25)/\langle 7 \rangle$ has genus $4$. The curve $X$ and its modular units have been studied by Lecacheux \cite{lecacheux:25} and Darmon \cite{darmon:X1_25}. Consider the following modular units:
\begin{align*}
u & = q \prod_{\substack{n \geq 1\\ n \equiv \pm 3,\pm 4 (25)}} (1-q^n) / \prod_{\substack{n \geq 1 \\ n \equiv \pm 2,\pm 11 (25)}} (1-q^n)\\
v & = q^{-1} \prod_{\substack{n \geq 1\\ n \equiv \pm 9,\pm 12 (25)}} (1-q^n) / \prod_{\substack{n \geq 1 \\ n \equiv \pm 6,\pm 8 (25)}} (1-q^n).
\end{align*}
Their minimal polynomial is given by
\begin{equation*}
P_{25}=y^2 x^4 + (y^3 + y^2) x^3 + (3y^3 - y^2 - 2y)x^2 + (y^4 - 4y^2 + y - 1)x - y^3.
\end{equation*}
This polynomial vanishes on the torus at the points $(x,y)=(\zeta,-\zeta)$ for each primitive $5$-th root of unity $\zeta$. These points are cusps: letting $\zeta_5=e^{2\pi i/5}$, we have
\begin{align*}
u(1/5) & =\zeta_5^2=-v(1/5) & u(-1/5) & =\zeta_5^{-2} = -v(-1/5)\\
u(2/5) & =\zeta_5=-v(2/5) & u(-2/5) & =\zeta_5^{-1} = -v(-2/5).
\end{align*}
The Deninger cycle associated to $P_{25}$ is given by $\gamma_{P_{25}}=\gamma_0+\gamma_1$ where $\gamma_0$ is a closed cycle and $\gamma_1 = \left\{\frac15,-\frac15\right\}+ \left\{-\frac25,\frac25\right\}$. Using \cite[Thm 1]{zudilin}, we get
\begin{equation*}
\int_{\gamma_1} \eta(u,v) = \frac{1}{4\pi} L(F,2)
\end{equation*}
where $F$ is a modular form of weight 2 and level 25. This time $F$ is a linear combination of newforms and Eisenstein series. Let $\eps : (\Z/25\Z)^\times \to \C^\times$ be the unique Dirichlet character such that $\eps(2)=\zeta_5$. A basis of eigenforms of $\Omega^1(X) \otimes \C$ is given by newforms $(f_a)_{a \in (\Z/5\Z)^\times}$ having Fourier coefficients in $\Q(\zeta_5)$ and forming a single Galois orbit. The newform $f_a$ has character $\eps^a$ and for any $\sigma \in \Gal(\Q(\zeta_5)/\Q)$, we have $\sigma(f_a) = f_{\chi(\sigma) a}$ where $\chi$ is the cyclotomic character. Moreover, let $\psi : (\Z/5\Z)^\times \to \C^\times$ be the Dirichlet character defined by $\psi(2)=i$. Then $W_{25} F$ has integral coefficients and is given by
\begin{equation*}
W_{25} F = -10 \operatorname{Tr}_{\Q(\zeta_5)/\Q} (\lambda f_1) - 25(1+i) E_2^{\psi,\overline{\psi}} - 25 (1-i) E_2^{\overline{\psi},\psi}
\end{equation*}
where $\lambda = 2\zeta_5+\zeta_5^{-1}+2\zeta_5^{-2}$ and $E_2^{\psi,\overline{\psi}}$ is the Eisenstein series defined by
\begin{equation*}
E_2^{\psi,\overline{\psi}} = \sum_{m,n=1}^\infty m \overline{\psi}(m) \psi(n) q^{mn}.
\end{equation*}
We may therefore expect $m(P_{25})$ being a linear combination of $L'(\psi,-1)$, $L'(\overline{\psi},-1)$ and $L'(f,0)$, where $f$ is a cusp form with rational Fourier coefficients. Indeed, we find numerically
\begin{equation}\label{eq mP25}
m(P_{25}) \stackrel{?}{=} L'(f,0) + \frac{1+2i}{5} L'(\overline{\psi},-1) + \frac{1-2i}{5} L'(\psi,-1)
\end{equation}
where
\begin{equation*}
f =\frac15 \operatorname{Tr}((2+\zeta_5+2\zeta_5^{-2}) f_1) = q + q^2-q^3-q^4-3q^5-2q^9+3q^{10}+4q^{11}+O(q^{12}).
\end{equation*}

\bibliographystyle{smfplain}   
\bibliography{references}

\end{document}